\documentclass[reqno,11pt]{amsart}

\usepackage{amssymb,amsmath}
\usepackage{mathrsfs}

\large\normalsize
\newtheorem{theorem}{Theorem}
\newtheorem{lemma}[theorem]{Lemma}

\newtheorem{definition}[theorem]{Definition}
\newtheorem{example}[theorem]{Example}
\newtheorem{remark}[theorem]{Remark}

\newcommand{\R}{\ensuremath{\mathbb{R}}}
\newcommand{\Z}{\ensuremath{\mathbb{Z}}}
\newcommand{\C}{\ensuremath{\mathbb{C}}}
\newcommand{\N}{\ensuremath{\mathbb{N}}}
\newcommand{\T}{\ensuremath{\mathbb{X}}}
\newcommand{\D}{\ensuremath{\mathbb{D}}}

\DeclareMathOperator{\diag}{diag}
\DeclareMathOperator{\subdiag}{subdiag}
\DeclareMathOperator{\supdiag}{superdiag}
\DeclareMathOperator{\im}{Im}
\DeclareMathOperator{\re}{Re}
\DeclareMathOperator{\range}{range}
\DeclareMathOperator{\trans}{T}
\DeclareMathOperator{\loc}{loc}
\DeclareMathOperator{\acloc}{AC_{loc}}

\numberwithin{equation}{section}
\numberwithin{theorem}{section}

\small\normalsize

\begin{document}

\title[Titchmarsh-Sims-Weyl Theory]{Titchmarsh-Sims-Weyl theory for complex Hamiltonian systems on Sturmian time scales}

\author[anderson]{Douglas R. Anderson} 
\address{Department of Mathematics and Computer Science, Concordia College, Moorhead, MN 56562 USA \\ visiting the School of Mathematics, The University of New South Wales Sydney 2052, Australia}
\email{andersod@cord.edu}


\begin{abstract} 
We study non-self-adjoint Hamiltonian systems on Sturmian time scales, defining Weyl-Sims sets, which replace the classical Weyl circles, and a matrix-valued $M-$function on suitable cone-shaped domains in the complex plane. Furthermore, we characterize realizations of the corresponding dynamic operator and its adjoint, and construct their resolvents. Even-order scalar equations and the Orr-Sommerfeld equation on time scales are given as examples illustrating the theory, which are new even for difference equations. These results unify previous discrete and continuous theories to dynamic equations on Sturmian time scales.
\end{abstract}

\maketitle\thispagestyle{empty}



\section{Introduction}\label{secintro}

Brown, Evans, and Plum \cite{b2} study Titchmarsh-Sims-Weyl theory for the complex, generally non-self-adjoint continuous Hamiltonian system
\begin{equation}\label{ms1.2}
 Jy'(t,\lambda)=\Big(\lambda A(t)+B(t)\Big)y(t,\lambda), \quad t\in[0,\infty),
\end{equation}
where $A$ and $B$ are $2n\times 2n$ matrix-valued functions with matrix weight function $A(t)\ge 0$, and
$J=\left(\begin{smallmatrix} 0_n & -I_n \\ I_n & 0_n \end{smallmatrix}\right)$. In a related work, Monaquel and Schmidt \cite{ms} introduce the uniformly discrete (the step size of the domain is a constant unit) counterpart to the continuous Hamiltonian system via
\begin{equation}\label{ms1.3}
 J{\bf\Delta}y(t,\lambda)=(\lambda A(t)+B(t))y(t,\lambda) \quad\text{for}\quad t\in[0,\infty)\cap\Z,
\end{equation}
where ${\bf\Delta}$ is a mixed right-and-left difference operator described via ${\bf\Delta}=\left(\begin{smallmatrix} \Delta & 0_n \\ 0_n & \nabla \end{smallmatrix}\right)$ for the forward difference operator $\Delta u(t)=u(t+1)-u(t)$ and the backward difference operator $\nabla u(t)=u(t)-u(t-1)$, and where again $A$ and $B$ are $2n\times 2n$ matrix-valued functions but with the assumption $A(t)>0$.  

We seek to extend \eqref{ms1.2}, \eqref{ms1.3} to Sturmian time scales (introduced in Ahlbrandt, Bohner, and Voepel \cite{abv}), thus unifying the  continuous \eqref{ms1.2} and discrete \eqref{ms1.3} non-self-adjoint theories in a single setting. As we do so, the robust nature of time-scale theory will offer more flexibility when discretizing \eqref{ms1.2}, for example allowing time-varying step sizes in the domain, than that represented by \eqref{ms1.3}. Just as in the translation of the continuous theory to the uniformly discrete case, the unification and extention of the two theories to Sturmian time scales require some care and provide unexpected difficulties. The first such issue is defining an appropriate time-scale dynamic operator that generalizes ${\bf\Delta}$ in \eqref{ms1.3} and accounts for the shifts at scattered domain points. To construct such an operator, we follow \cite{ms} by utilizing a partial left-shift operator applied to $y$ defined on $[t_0,\infty)_\T$ for some $t_0\in\T$ by 
\begin{equation}\label{leftshift}
 \widehat{y}(t):=\begin{pmatrix} y_1(t) \\ y_2^\rho(t) \end{pmatrix}, \quad y_1:[t_0,\infty)_\T\rightarrow\C^{n}, \quad y_2:[\rho(t_0),\infty)_\T\rightarrow\C^{n}.
\end{equation}
Then we introduce the complex Hamiltonian dynamic system on Sturmian time scales given by
\begin{equation}\label{maineq}
 J\widehat{y}^\Delta(t)=\Big(\lambda A(t)+B(t)\Big) y(t), \quad t\in[t_0,\infty)_\T:=[t_0,\infty)\cap\T, \quad J=\left(\begin{smallmatrix} 0_n & -I_n \\ I_n & 0_n \end{smallmatrix}\right),
\end{equation}
where we assume that the Hermitian weight function $A:[t_0,\infty)_\T\rightarrow\C^{2n,2n}$ and the (generally) non-Hermitian coefficient function $B:[t_0,\infty)_\T\rightarrow\C^{2n,2n}$ satisfy the block forms
\begin{equation}\label{ABdef}
 A(t) = \begin{pmatrix} A_1(t) & 0_n \\ 0_n & A_2(t) \end{pmatrix} \ge 0 \quad\text{and}\quad B(t) = \begin{pmatrix} B_1(t) & B_2(t) \\ B_3(t) & B_4(t) \end{pmatrix} 
\end{equation}
for $n\times n$ rd-continuous complex matrices $A_1$, $A_2$, $B_1$, $B_2$, $B_3$, and $B_4$ such that
\begin{equation}\label{msA1}
  E_2(t):=\Big(I_n+\mu(t) B_2(t)\Big)^{-1} \quad\text{and}\quad \Big(I_n+\mu(t) B_3(t)\Big)^{-1} \quad\text{exist.}
\end{equation} 
Our positive semi-definite assumption in \eqref{ABdef} on $A$ weakens the positive definite assumption made on $A$ in the discrete case \cite{ms}; in compensation, we make a certain definiteness assumption below in \eqref{bep3.14} that mirrors the continuous case \cite[(3.14)]{b2}. Using standard notation \cite{sh}, $\T$ is a nonempty unbounded closed subset of the set of real numbers $\R$ such that the left jump operator $\rho$ and right jump operator $\sigma$ given by
$$ \rho(t)=\sup\{s\in\T:s<t\} \quad\text{and}\quad \sigma(t)=\inf\{s\in\T: s>t\} $$
satisfy the Sturmian \cite{abv} time-scale condition
\begin{equation}\label{sturmfact}
 \sigma(\rho(t)) = \rho(\sigma(t)) = t, \quad t\in[t_0,\infty)_\T,
\end{equation}
with the compositions $y\circ\rho$ and $z\circ\sigma$ denoted by $y^\rho$ and $z^\sigma$, respectively; the graininess functions are defined by $\mu(t)=\sigma(t)-t$ and $\nu(t)=t-\rho(t)$; the delta derivative of $y$ at $t\in\T$, denoted $y^\Delta(t)$, and the nabla derivative of $y$ at $t\in\T$, denoted $y^\nabla(t)$, are the vectors (provided they exist) given by, respectively,
$$ y^\Delta(t):=\lim_{s\rightarrow t}\frac{y^\sigma(t)-y(s)}{\sigma(t)-s} \quad\text{and}\quad y^\nabla(t):=\lim_{s\rightarrow t}\frac{y^\rho(t)-y(s)}{\rho(t)-s}. $$
Note that if $\T=\R$ we have
$$ J\widehat{y}^\Delta(t) = Jy'(t) $$
as in \eqref{ms1.2}, while if $\T=\Z$ we have
$$ J\widehat{y}^\Delta(t) = J\Delta\widehat{y}(t) = J\Delta \left(\begin{smallmatrix} y_1(t) \\ y_2(t-1) \end{smallmatrix}\right) = J \left(\begin{smallmatrix} \Delta y_1(t) \\ \nabla y_2(t) \end{smallmatrix}\right) = J{\bf\Delta}y(t) $$
as in \eqref{ms1.3}, giving credibility to system \eqref{maineq} as a unifying vehicle for this theory. To recover the case $\T=\R$, just set $\mu(t)=0$ and $\sigma(t)=\rho(t)=t$, and to recover the case $\T=\Z$, set $\mu(t)=1$, $\sigma(t)=t+1$, and $\rho(t)=t-1$.

For $A$ given in \eqref{ABdef}, let $L_A^2(t_0,\infty)_\T$ denote the Hilbert space of measurable $\C^{2n}-$valued functions $y$ for which the delta integral exists and satisfies
\begin{equation}\label{normdef}
 \|y\|_A^2:=\int_{t_0}^{\infty} y^*(t)A(t)y(t)\Delta t < \infty.
\end{equation}
Vector functions $x,y\in L_A^2(t_0,\infty)_\T$ are said to be $A-$square integrable, with the scalar product defined via
\begin{equation}\label{inprod}
 (x,y)_A:=\int_{t_0}^{\infty}y^*(t)A(t)x(t)\Delta t, \quad x,y\in L_A^2(t_0,\infty)_\T. 
\end{equation}
As $A$ in \eqref{ABdef} may be singular, the inner product for $L_A^2(t_0,\infty)_\T$ in \eqref{inprod} may not be positive. To account for this, we introduce the following quotient space. For $x,y\in L_A^2(t_0,\infty)_\T$, $x$ and $y$ are said to be equal iff $\|x-y\|_A=0$. In this context $L_A^2(t_0,\infty)_\T$ is an inner product space with inner product \eqref{inprod}. In addition, a $2n\times n$ matrix is $A-$square integrable if and only if each of its columns is $A-$square integrable. The same terminology will be used for other $2n\times 2n$ matrix weight functions. 

Next we define the linear vector function space
\begin{equation}\label{Ddef}
 \D:=\left\{y=(y_1,y_2)^{\trans}\in L_A^2(t_0,\infty)_\T\Big|\; y_1,y_2^\rho:[t_0,\infty)_\T\rightarrow\C^n \; \text{are delta differentiable} \right\}. 
\end{equation}
Then $y$ is a solution of \eqref{maineq} if and only if $y\in\D$ and $y$ satisfies \eqref{maineq}. With \eqref{sturmfact} in mind, we point out that in \eqref{Ddef} the assumption $y_2^{\rho}$ is delta differentiable is equivalent to the assumption that $y_2$ is delta differentiable. For more on time scales generally, see Bohner and Peterson \cite{bp1,bp2}. 

In \cite{b1}, Brown, Evans, McCormack, and Plum study the equation
$$ -(py')'+qy=\lambda w y, $$
allowing the potentials to be complex-valued functions, and in so doing relax some of the conditions used by Sims \cite{sims}; see \eqref{bep3.23} below. Atkinson \cite{fva}, Hinton and Shaw \cite{hs1,hs2}, and Krall \cite{krall} all worked on \eqref{ms1.2} in the case of symmetric coefficients for $\T=\R$. In the discrete symmetric case, Clark and Gesztesy \cite{cg1,cg2} presented a Weyl-Titchmarsh theory for a version of \eqref{ms1.3}, while Shi \cite{shi2} did the same for a version of \eqref{alteq} given below. For more on related discrete theory for Hamiltonian systems see Ahlbrandt \cite{ahl}, Ahlbrandt and Peterson \cite{ap}, Bohner, Do\v{s}l\'{y}, and Kratz \cite{bdk}, Shi \cite{shi1}, Shi and Wu \cite{sw}, and Sun, Shi, and Chen \cite{ssc}. Preliminary work on Hamiltonian systems on time scales includes Ahlbrandt, Bohner, and Ridenhour \cite{AlBoRi}, Anderson \cite{and}, Anderson and Buchholz \cite{ab}, and Hilscher \cite{Hilscher}. For some specific examples where non-self-adjoint problems may arise see Chandrasekhar \cite{chandra}. 

Our notation and organization of the discussion are obviously largely based on the continuous \cite{b2} and discrete \cite{ms} cases. The key to bringing these two theories together is finding an operator with domain inside a weighted Hilbert space such that existence and uniqueness of solutions can be shown for initial value problems, and an integration by parts formula holds. In the next section we will give a detailed analysis of these issues on Sturmian time scales. The Sturmian assumption in \eqref{sturmfact} may from one point of view seem unduly restrictive; from another it can be seen as surprising that the discrete and continuous theories can be combined at all, as there is no a priori physical reason why this theory should exist on all possible time scales, especially pathological ones. Assumption \eqref{sturmfact} merely requires that all points be dense from both sides or scattered from both sides.  With this supposition in place, we will see how unexpectedly harmonious the continuous and discrete cases can be made when treated concurrently in this context.


\section{The homogeneous system}\label{homsys}

We begin our analysis of the linear Hamiltonian system \eqref{maineq} in this section with an existence and uniqueness result.


\begin{theorem}[Existence and Uniqueness]\label{thmeu}
Assume \eqref{msA1}. Then the linear Hamiltonian system \eqref{maineq} with initial condition 
\begin{equation}\label{msic} 
 \widehat{y}(t_0) = \left(\begin{smallmatrix} y_1(t_0) \\ y_2^\rho(t_0) \end{smallmatrix}\right) 
\end{equation}
has a unique solution $\widehat{y}=(y_1,y_2^\rho)^{\trans}$ with $y\in\D$.
\end{theorem}

\begin{proof}
Let $\widehat{y}$ be given as in \eqref{leftshift}. For any given $\lambda\in\C$, using \eqref{sturmfact} and the assumptions on the block forms of $J$, $A$, and $B$, we can rewrite \eqref{maineq} as the pair of $n$-vector equations
$$ \begin{cases} 
   &y_1^\Delta(t) = B_3(t)y_1(t) + \big(\lambda  A_2(t)+B_4(t)\big)(y_2^\rho)^\sigma(t), \\
   &(y_2^\rho)^\Delta(t) = -\big(\lambda  A_1(t)+B_1(t)\big)y_1(t) - B_2(t)(y_2^\rho)^\sigma(t). \end{cases} $$
Using the simple useful formula $y_2=(y_2^\rho)^\sigma=(y_2^\rho)+\mu(y_2^\rho)^\Delta$, we also have (suppressing the $t$)
\begin{equation}\label{ms1.9}
 \begin{cases} 
   &y_1^\Delta = \Big(B_3-\mu(\lambda A_2+B_4)E_2(\lambda A_1+B_1)\Big)y_1 + (\lambda A_2+B_4)E_2(y_2^\rho), \\
   &(y_2^\rho)^\Delta = -E_2\big(\lambda A_1+B_1\big)y_1 - E_2B_2(y_2^\rho), \end{cases}
\end{equation}
where we have taken $E_2$ as in \eqref{msA1}. Thus we may also view solutions $y=(y_1,y_2)^{\trans}$ of \eqref{maineq} and \eqref{ms1.9} via \eqref{leftshift} as solutions of 
\begin{equation}\label{delsym}
 \widehat{y}^\Delta(t)=\mathcal{K}(t,\lambda)\widehat{y}(t), \quad \mathcal{K}(\cdot,\lambda):=-J(\lambda A+B)H,
\end{equation}
where on $[t_0,\infty)_\T$ we use $E_2=(I_n+\mu B_2)^{-1}$ and
\begin{equation}\label{ms2.2}
 H:=\begin{pmatrix} I_n & 0_n \\ -\mu E_2(\lambda A_1+B_1) & E_2 \end{pmatrix}, \quad\text{with}\quad -J(\lambda A+B)=\begin{pmatrix} B_3 & \lambda A_2+B_4 \\ -(\lambda A_1+B_1) & -B_2 \end{pmatrix},
\end{equation}
since $E_2B_2=B_2E_2$. Directly from the definition of $\mathcal{K}(\cdot,\lambda)$ in \eqref{delsym} and \eqref{ms2.2} we have that
$$ I_{2n}+\mu\mathcal{K}(\cdot,\lambda) = \begin{pmatrix} I_n+\mu B_3 & \mu(\lambda A_2+B_4) \\ 0_n & I_n \end{pmatrix} \begin{pmatrix} I_n & 0_n \\ -\mu E_2(\lambda A_1+B_1) & E_2\end{pmatrix}, $$
so that $I_{2n}+\mu(t)\mathcal{K}(t,\lambda)$ is invertible by \eqref{msA1} and thus $\mathcal{K}(\cdot,\lambda)$ is regressive. By \cite[Theorem 5.8]{bp1}, the matrix equation $\widehat{y}^\Delta=\mathcal{K}(\cdot,\lambda)\widehat{y}$ with initial condition \eqref{msic} has a unique solution $\widehat{y}$. It follows that \eqref{maineq}, \eqref{leftshift} with initial condition \eqref{msic} has a unique solution $y\in\D$.
\end{proof}


\begin{theorem}[Green's Formula]\label{ibp}
For $y,z\in\D$ and $a,b\in[t_0,\infty)_\T$ with $b>a$ we have
$$ \int_{a}^{b} \left[z^*J\widehat{y}^\Delta - \left(J\widehat{z}^\Delta\right)^*y \right](t)\Delta t 
         = \widehat{z}^*(b)J\widehat{y}(b) - \widehat{z}^*(a)J\widehat{y}(a). $$
\end{theorem}

\begin{proof}
Using \eqref{sturmfact} as we expand out the integrand, we have (suppressing the variable $t$)
\begin{eqnarray*} 
 z^*J\widehat{y}^\Delta &=& -z_1^*(y_2^\rho)^\Delta + (z_2^{\rho})^{\sigma*}y_1^\Delta, \\
 \left(J\widehat{z}^\Delta\right)^*y &=& -(z_2^{\rho})^{\Delta*}y_1 + z_1^{\Delta*}(y_2^\rho)^\sigma, 
\end{eqnarray*}
so that when we subtract the second from the first, we obtain
\begin{eqnarray*}
 z^*J\widehat{y}^\Delta - \left(J\widehat{z}^\Delta\right)^*y 
 &=& -z_1^*(y_2^\rho)^\Delta + (z_2^{\rho})^{\sigma*}y_1^\Delta + (z_2^{\rho})^{\Delta*}y_1 - z_1^{\Delta*}(y_2^\rho)^\sigma \\
 &=& -(z_1^*y_2^\rho)^\Delta+(z_2^{\rho*}y_1)^\Delta \\
 &=& \left[\widehat{z}^*J\widehat{y}\right]^\Delta.
\end{eqnarray*}
The result follows from the fundamental theorem of calculus.
\end{proof}

If $y\in\D$ is a solution of \eqref{maineq}, then it is related to $\widehat{y}$ via
\begin{equation}\label{ms2.1}
 y(t) = H(t)\widehat{y}(t), \quad t\in[t_0,\infty)_\T,
\end{equation}
for $H$ given in \eqref{ms2.2}. As a result we may write \eqref{maineq} as the equivalent equation \eqref{delsym}, in other words as
\begin{equation}\label{ms2.3}
 J\widehat{y}^\Delta(t) = \big(\lambda A(t)+B(t)\big)H(t)\widehat{y}(t).
\end{equation}
Moreover, the formal adjoint of \eqref{maineq} takes the form
\begin{equation}\label{ms2.4}
 J\widehat{z}^\Delta(t)=\Big(\overline{\lambda} A(t)+B^*(t)\Big)z(t), \quad \lambda\in\C, \quad t\in[t_0,\infty)_\T;
\end{equation}
under the hypotheses of Theorem $\ref{thmeu}$, an existence and uniqueness result holds for \eqref{ms2.4}. A solution $z$ of \eqref{ms2.4} is related to $\widehat{z}$ by the transformation
\begin{equation}\label{ms2.5}
 z(t) = \widetilde{H}(t)\widehat{z}^{\;\sigma}(t), \quad t\in[t_0,\infty)_\T, 
\end{equation}
where on $[t_0,\infty)_\T$ we have taken $E_2$ as in \eqref{msA1} and
\begin{equation}\label{Htilde}
 \widetilde{H}:=\begin{pmatrix} E_2^* & -\mu E_2^*\Big(\overline{\lambda} A_2+B_4^*\Big) \\ 0_n & I_n \end{pmatrix}, \quad \quad \widetilde{H}^{-1}=\begin{pmatrix} I+\mu B_2^* & \mu\Big(\overline{\lambda} A_2+B_4^*\Big) \\ 0_n & I_n \end{pmatrix}. 
\end{equation}
It follows for $\lambda\in\C$ and $t\in[t_0,\infty)_\T$ that \eqref{ms2.4} is equivalent to the equation
\begin{equation}\label{ms2.6}
 J\widehat{z}^\Delta(t) = \Big(\overline{\lambda} A(t)+B^*(t)\Big)\widetilde{H}(t)\widehat{z}^{\;\sigma}(t)=H^*(t)\Big(\overline{\lambda} A(t)+B^*(t)\Big)\widehat{z}^{\;\sigma}(t).
\end{equation}


\begin{definition}
A matrix $\widehat{Y}:[t_0,\infty)_\T\rightarrow\C^{2n,2n}$ is called a fundamental system of \eqref{ms2.3} if and only if its columns are linearly independent solutions of \eqref{ms2.3}, if and only if $\widehat{Y}(t)$ has rank $2n$ for some $t\in[t_0,\infty)_\T$. In this case the columns of $Y(t)=H(t)\widehat{Y}(t)$ form a fundamental system (system of linearly independent solutions) of \eqref{maineq}, where $H$ is given in \eqref{ms2.2}.
\end{definition}


\begin{remark}\label{rmk1.4}
In the development below, let the matrix $\widehat{Y}(t)=\left(\widehat{\theta}(t)|\widehat{\phi}(t)\right)$ for $t\in[t_0,\infty)_\T$ be the fundamental system of \eqref{ms2.3} satisfying the initial condition
\begin{equation}\label{ms2.7}
 \widehat{Y}(t_0)=J.
\end{equation}
It follows from \eqref{ms2.1} that $\widehat{Y}(t)=\left(\widehat{\theta}(t)|\widehat{\phi}(t)\right)$ for $t\in[t_0,\infty)_\T$ is the fundamental system of \eqref{maineq} satisfying
$$ Y(t_0) = H(t_0)J = \begin{pmatrix} 0_n & -I_n \\ \Big(I+\mu(t_0)B_2(t_0)\Big)^{-1} & \mu(t_0)\Big(I+\mu(t_0)B_2(t_0)\Big)^{-1}\Big(\lambda A_1(t_0)+B_1(t_0)\Big) \end{pmatrix}. $$
In a similar fashion take $\widehat{Z}(t)=\left(\widehat{\eta}(t)|\widehat{\chi}(t)\right)$ for $t\in[t_0,\infty)_\T$ to be the fundamental system of \eqref{ms2.6} satisfying $\widehat{Z}(t_0)=J$, so that by \eqref{ms2.5} and \eqref{ms2.6} we have $Z(t_0)=[\mu J(\overline{\lambda}A+B^*)+\widetilde{H}^{-1}]^{-1}J$, i.e.
$$ Z(t_0) = \begin{pmatrix} 0_n & -I_n \\ \Big(I_n+\mu(t_0)B_3^*(t_0)\Big)^{-1} & \mu(t_0)\Big(I_n+\mu(t_0)B_3^*(t_0)\Big)^{-1}\Big(\overline{\lambda}A_1(t_0)+B_1^*(t_0)\Big) \end{pmatrix}. $$
\end{remark} 


\begin{lemma}\label{mslem2.4}
For the fundamental systems $\widehat{Y}$ and $\widehat{Z}$ of \eqref{ms2.3} and \eqref{ms2.6}, respectively, the equality 
$$ \widehat{Z}(t)=-J(\widehat{Y}^{-1})^*(t)J $$ holds for all $t\in[t_0,\infty)_\T$.
\end{lemma}

\begin{proof}
Let $\widehat{U}(t):=-J(\widehat{Y}^{-1})^*(t)J$ for all $t\in[t_0,\infty)_\T$. As $\widehat{Y}(t_0)=J$, we see that
$$ \widehat{U}(t_0) = -J(\widehat{Y}^{-1})^*(t_0)J = -JJJ = J = \widehat{Z}(t_0). $$
If we can show that $\widehat{U}$ solves \eqref{ms2.6}, then by uniqueness we will have $\widehat{U}=\widehat{Z}$. By the product rule,
$$ 0_{2n} = \left(\widehat{Y}^{-1}\widehat{Y}\right)^\Delta(t) = (\widehat{Y}^{-1})^{\sigma}(t)\widehat{Y}^\Delta(t) + (\widehat{Y}^{-1})^{\Delta}(t)\widehat{Y}(t); $$
from this and \eqref{ms2.3} we have that
$$ (\widehat{Y}^{-1})^{\Delta}(t) = -(\widehat{Y}^{\sigma})^{-1}(t)\widehat{Y}^\Delta(t)\widehat{Y}^{-1}(t) = (\widehat{Y}^{\sigma})^{-1}(t)J\big(\lambda A(t)+B(t)\big)H(t). $$
Putting it all together we get
\begin{eqnarray*}
  J\widehat{U}^\Delta(t) &=& J\left(-J(\widehat{Y}^{-1})^*J\right)^\Delta(t) = (\widehat{Y}^{-1})^{*\Delta}(t)J = (\widehat{Y}^{-1})^{\Delta*}(t)J \\
  &=& H^*(t)\big(\lambda A(t)+B(t)\big)^*J^*(\widehat{Y}^{\sigma})^{-1*}(t)J \\
  &=& H^*(t)\big(\lambda A(t)+B(t)\big)^*\widehat{U}^{\sigma}(t),
\end{eqnarray*}
so that $\widehat{U}$ solves \eqref{ms2.6}. Thus $\widehat{Z}(t)=\widehat{U}(t)=-J(\widehat{Y}^{-1})^*(t)J$ for all $t\in[t_0,\infty)_\T$.
\end{proof}

Throughout the paper, we will exhibit explicit dependence on $\lambda$ only when necessary. We end this section with the following remark. 


\begin{remark}
As seen above in Theorems $\ref{thmeu}$ and $\ref{ibp}$, the assumption that $\T$ is Sturmian plays a decisive role in securing existence and uniqueness of solutions, and an integration by parts formula. In this remark, we show that there is no getting around this assumption. For example, instead of system \eqref{maineq}, consider the alternative system
\begin{equation}\label{alteq}
 Jy^\nabla(t)=\Big(\lambda A(t)+B(t)\Big)\widehat{y}(t), \quad t\in[t_0,\infty)_\T, \quad J=\begin{pmatrix} 0_n & -I_n \\ I_n & 0_n \end{pmatrix},
\end{equation}
where we have $\widehat{y}$ on the right-hand side, the nabla derivative is used on the left, and we assume
\begin{equation}\label{msA1alt}
  E_2'(t):=\Big(I_n-\nu(t) B_2(t)\Big)^{-1} \quad\text{and}\quad \Big(I_n-\nu(t) B_3(t)\Big)^{-1} \quad\text{exist}
\end{equation} 
in place of \eqref{msA1}. System \eqref{alteq} may also be viewed as a generalization of \eqref{ms1.2} and \eqref{ms1.3}. As in \eqref{delsym}, we can rewrite \eqref{alteq} as the system 
\begin{equation}\label{nabsym}
 y^\nabla(t)=\mathcal{K}_1(t,\lambda)y(t),  \quad \mathcal{K}_1(\cdot,\lambda):=-J(\lambda A+B)H_1, \quad  H_1:=\begin{pmatrix} I_n & 0_n \\ \nu E_2'(\lambda A_1+B_1) & E_2' \end{pmatrix}.
\end{equation}
Directly from the definition of $\mathcal{K}_1(\cdot,\lambda)$ in \eqref{nabsym} we have that
$$ I_{2n}-\nu\mathcal{K}_1(\cdot,\lambda) = \begin{pmatrix} I_n-\nu B_3 & -\nu(\lambda A_2+B_4) \\ 0_n & I_n \end{pmatrix} H_1, $$
so that $I_{2n}-\nu(t)\mathcal{K}_1(t,\lambda)$ is invertible by \eqref{msA1alt}, $\mathcal{K}_1(\cdot,\lambda)$ is $\nu$-regressive, and the matrix equation $y^\nabla=\mathcal{K}_1(\cdot,\lambda)y$ with proper initial condition has a unique solution $y$. It follows that \eqref{alteq} has a unique solution. Additionally, we have the integration by parts formula (Green's formula)
$$ \int_{a}^{b} \left[\widehat{z}^*Jy^\nabla - \left(Jz^\nabla\right)^*\widehat{y} \right](t)\nabla t=z^*(b)Jy(b)-z^*(a)Jy(a). $$
To this point in the remark the analysis is valid on general time scales. However, the scalar product in \eqref{inprod} is now replaced by
\begin{equation}\label{altprod}
 (x,y)_A:=\int_{t_0}^{\infty}\widehat{y}^*(t)A(t)\widehat{x}(t)\nabla t, \quad x,y\in L_A^2(t_0,\infty)_\T. 
\end{equation}
To show that this defines an inner product on $L_A^2(t_0,\infty)_\T$, we would need the hat operator $\widehat{\cdot}$ in \eqref{leftshift} to be invertible, which is only possible if \eqref{sturmfact} holds, i.e. on Sturmian time scales. In summary, to unify \eqref{ms1.2} and \eqref{ms1.3} on time scales, systems equivalent to \eqref{maineq} or \eqref{alteq} must be used to account for the shifts in the discrete case \cite{ms}. For those systems to admit the existence and uniqueness of solutions to initial value problems, an integration by parts formula, and a matrix weighted scalar product in a Hilbert space, the Sturmian assumption \eqref{sturmfact} must used somewhere.
\end{remark}


\section{Weyl-Sims nesting sets}\label{WSnest}

Weyl-Sims sets $D(t,\lambda)$, $t\in[t_0,\infty)_\T$, are the analogue of the classical Weyl circles. Here the spectral parameter $\lambda$ varies in a set $\Lambda(\lambda_0,\mathscr{U}_{2n})\subset\C$, a cone-shaped set defined to parallel the construction in \cite{b2}, which takes the role of Sims' rotated half-planes. The class of matrices $\mathscr{U}_{2n}$ is one of many possible describing this rotation. The key result is the nesting property of the Weyl-Sims sets (see Theorem \ref{msthm3.2} below). It will follow that there is a limit set $D(\infty,\lambda)$ with the property that for any $l\in D(\infty,\lambda)$, the Weyl solution $\psi=\theta+\phi l$ is square integrable with respect to a certain matrix weight function $W$, while an analogous statement holds for the adjoint equation. We will conclude this section by noting conditions which imply $A-$square integrability of the Weyl solution. 

To begin, choose $\mathscr{U}\in\C^{n,n}$ invertible and define
\begin{equation}\label{Udef}
 \mathscr{U}_{2n} = \begin{pmatrix} \mathscr{U} & 0_n \\ 0_n & -\mathscr{U}^* \end{pmatrix}. 
\end{equation}
Note that $\mathscr{U}_{2n}J$ is Hermitian and has precisely $n$ positive and $n$ negative eigenvalues. Indeed, if $(\lambda,v)$ is an eigenpair associated with $\mathscr{U}_{2n}J$, then 
$$ (-\lambda,w), \quad\text{where}\quad w = \begin{pmatrix} 0_n  & -\mathscr{U} \\ \mathscr{U}^* & 0_n \end{pmatrix}v, $$
is also an eigenpair associated with $\mathscr{U}_{2n}J$. Moreover, since $J\mathscr{U}_{2n}^*=-\mathscr{U}_{2n}J$, a use of Theorem \ref{ibp}, \eqref{ms2.1} and \eqref{ms2.3} yields that the fundamental system $\widehat{Y}=\left(\widehat{\theta}|\widehat{\phi}\right)$ satisfies
\begin{eqnarray}
 \widehat{Y}^*(t)\mathscr{U}_{2n}J\widehat{Y}(t) - \widehat{Y}^*(t_0)\mathscr{U}_{2n}J\widehat{Y}(t_0)
 &=& \int_{t_0}^{t} \left\{Y^*\left(\mathscr{U}_{2n}J\widehat{Y}^\Delta\right)+\left(\mathscr{U}_{2n}J\widehat{Y}^\Delta\right)^*Y\right\}(s)\Delta s \nonumber \\
 &=& \int_{t_0}^{t} Y^*\left\{\mathscr{U}_{2n}(\lambda A+B)+\left[\mathscr{U}_{2n}(\lambda A+B)\right]^*\right\}Y(s)\Delta s \nonumber \\
 &=& 2 \int_{t_0}^{t} Y^*(s)W(s,\lambda)Y(s)\Delta s, \label{ms3.1}
\end{eqnarray}
where 
\begin{equation}\label{Wdef}
 W(t,\lambda):=\re\left[\mathscr{U}_{2n}(\lambda A(t)+B(t))\right].
\end{equation}
Note that by using the notation $W(t,\lambda)$ in \eqref{Wdef} we mimic the discrete case \cite[(3.1)]{ms}, which uses $W_k(\lambda)$ for $k\in\N$; in the continuous case \cite[(3.5)]{b2} they use the notation $C_{\lambda}(x)$ for $x\in\R$. For $M\in\C^{n,n}$ we take as its norm $\|M\|$ the largest eigenvalue of $(M^*M)^{1/2}$, and define its real and imaginary parts as
\begin{equation}\label{realpart}
 \re[M]=\frac{1}{2}(M+M^*) \quad\text{and}\quad \im[M]=\frac{1}{2i}(M-M^*). 
\end{equation}


\begin{definition}
Let $\lambda_0\in\C$, and let $\mathscr{U}_{2n}$ be given as above in \eqref{Udef}. Then $(\lambda_0,\mathscr{U}_{2n})$ is called an admissible pair for \eqref{maineq}, and denoted $(\lambda_0,\mathscr{U}_{2n})\in\mathscr{S}$, if and only if
\begin{equation}\label{ms3.2}
 W(t,\lambda_0) = \re\left[\mathscr{U}_{2n}\big(\lambda_0 A(t)+B(t)\big)\right] \ge 0, \quad t\in[t_0,\infty)_\T.
\end{equation}
In this case, we define the set
\begin{equation}\label{ms3.3}
  \Lambda(\lambda_0,\mathscr{U}_{2n}):=\left\{\lambda\in\C: \text{for some} \; \delta > 0, \re[(\lambda-\lambda_0)\mathscr{U}_{2n}A(t)]\ge \delta\mathscr{U}_{2n}A(t)\mathscr{U}_{2n}^*\;\forall\;t\in[t_0,\infty)_\T \right\}.
\end{equation}
Then the Weyl-Sims sets for \eqref{maineq} are defined for $\lambda\in\Lambda(\lambda_0,\mathscr{U}_{2n})$ and $\widehat{Y}=\left(\widehat{\theta}|\widehat{\phi}\right)$ as
\begin{equation}\label{ms3.4}
 D(t,\lambda):=\left\{l\in\C^{n,n}:\left(\widehat{\theta}(t)+\widehat{\phi}(t)l\right)^*\mathscr{U}_{2n}J\left(\widehat{\theta}(t)+\widehat{\phi}(t)l\right)\le 0 \right\}.
\end{equation}
\end{definition}
\noindent Prior to proving the nesting property of the Weyl-Sims sets (see Theorem \ref{msthm3.2} below), we need the following. Since
\begin{equation}\label{ms3.5}
 W(t,\lambda) \ge \delta \mathscr{U}_{2n}A(t)\mathscr{U}_{2n}^* \ge 0 \quad\text{for}\quad t\in[t_0,\infty)_\T \quad\text{and}\quad \lambda\in\Lambda(\lambda_0,\mathscr{U}_{2n}),
\end{equation}
with $A\ge 0$, in addition to \eqref{ms3.2} we require the following definiteness condition: for any $\lambda\in\Lambda(\lambda_0,\mathscr{U}_{2n})$ and $\zeta\in\C^n$, 
\begin{equation}\label{bep3.14}
 W(t,\lambda)\phi(t)\zeta=0 \quad\text{for all}\quad t\in[t_0,\infty)_\T \quad\implies\quad \zeta=0. 
\end{equation}
Setting 
$$ \widehat{\theta}(t)=\begin{pmatrix} \theta_1(t) \\ \theta_2^\rho(t) \end{pmatrix} \quad\text{and}\quad \widehat{\phi}(t)=\begin{pmatrix} \phi_1(t) \\ \phi_2^\rho(t) \end{pmatrix}, $$
where the blocks $\theta_1(t)$, $\theta_2(t)$, $\phi_1(t)$, $\phi_2(t)$ are $\C^{n,n}-$valued matrices, we write
\begin{equation}\label{ms3.6}
 \widehat{Y}^*(t)\mathscr{U}_{2n}J\widehat{Y}(t) = 2\begin{pmatrix} S(t) & T(t) \\ T^*(t) & P(t) \end{pmatrix},
\end{equation}
where
\begin{eqnarray*}
 S(t,\lambda) &=& -\re\left[\theta_1^*(t)\mathscr{U}\theta_2^\rho(t)\right], \\
 T(t,\lambda) &=& -\frac{1}{2}\left[\theta_1^*(t)\mathscr{U}\phi_2^\rho(t) + \theta_2^{\rho*}(t)\mathscr{U}^*\phi_1(t)\right], \\
 P(t,\lambda) &=& -\re\left[\phi_1^*(t)\mathscr{U}\phi_2^\rho(t)\right].
\end{eqnarray*}
From the initial condition in \eqref{ms2.7} we see that $P(t_0,\lambda)=0$. We will show in Lemma \ref{mslem3.1} below that \eqref{bep3.14} implies $P(t,\lambda)>0$ for $t\in\T$ sufficiently large. For $t\ge t_1=t_1(\lambda)$ given below in Lemma \ref{mslem3.1}, we employ the notation
\begin{equation}\label{circrad}
 \mathscr{C}(t,\lambda):=-\left(P^{-1}T^*\right)(t,\lambda), \quad \mathscr{R}(t,\lambda):=\left(TP^{-1}T^*-S\right)(t,\lambda). 
\end{equation}
The proof of the following lemma is the same, with slight modifications, as in the discrete \cite[Lemma 3.1]{ms} and continuous \cite[Lemma 3.5]{b2} cases. 


\begin{lemma}\label{mslem3.1}
For $\lambda\in\Lambda(\lambda_0,\mathscr{U}_{2n})$ for some $(\lambda_0,\mathscr{U}_{2n})\in\mathscr{S}$, there exists some $t_1=t_1(\lambda)\in[t_0,\infty)_\T$ such that
\begin{enumerate}
 \item $P(t,\lambda)$ is non-decreasing in $t$, $P(t,\lambda)\ge 0$, and, for $t\ge t_1$, $P(t,\lambda) > 0$,
 \item $D(t,\lambda)\ne \emptyset$ for $t\ge t_1$,
 \item for $t\in[t_1,\infty)_\T$, $\mathscr{R}(t,\lambda)$ is non-increasing in $t$, and $\mathscr{R}(t,\lambda)>0$.
\end{enumerate}
\end{lemma}

If we then multiply \eqref{ms3.6} on the left by $(I_n,l^*)$, and on the right by $\left(\begin{smallmatrix} I_n \\ l \end{smallmatrix}\right)$, we get the expression
$$ \left(\widehat{\theta}+\widehat{\phi}\;l\right)^*(t) \mathscr{U}_{2n}J \left(\widehat{\theta}+\widehat{\phi}\;l\right)(t) = 2 \left[l^*Pl+Tl+l^*T^*+S\right](t), \quad t\in[t_0,\infty)_\T. $$
It follows that
\begin{equation}\label{ms3.7}
 D(t,\lambda) = \left\{l\in\C^{n,n}: (l-\mathscr{C}(t,\lambda))^*P(t,\lambda)(l-\mathscr{C}(t,\lambda)) \le \mathscr{R}(t,\lambda)\right\}.
\end{equation}
Again as in both the discrete and continuous cases (see \cite[(3.7)]{ms} and \cite[Lemma 3.5(iii)]{b2}), we have that $\mathscr{R}(t,\lambda)\ge 0$ with $\mathscr{R}(t,\lambda) > 0$ for $t\ge t_1(\lambda)$, so that
\begin{equation}\label{ms3.8}
 D(t,\lambda) = \left\{l\in\C^{n,n}: l=\mathscr{C}(t,\lambda)+P^{-1/2}(t,\lambda)V\mathscr{R}^{1/2}(t,\lambda)\;\text{for some}\; V\in\C^{n,n}\; \text{with}\; V^*V\le I_n\right\},
\end{equation}
where $V$ can be taken as $V=P^{1/2}(t,\lambda)(l-\mathscr{C}(t,\lambda))\mathscr{R}^{-1/2}(t,\lambda)$.


\begin{theorem}\label{msthm3.2}
For $\lambda\in\Lambda(\lambda_0,\mathscr{U}_{2n})$ for some $(\lambda_0,\mathscr{U}_{2n})\in\mathscr{S}$, we have
\begin{enumerate}
 \item $D(t,\lambda)\subseteq D(\tau,\lambda)$ for $t,\tau\in[t_0,\infty)_\T$ with $t>\tau$,
 \item $D(t,\lambda)-\mathscr{C}(t,\lambda) \subseteq D(\tau,\lambda)-\mathscr{C}(\tau,\lambda)$ for $t,\tau\in[t_0,\infty)_\T$ with $t>\tau$, 
 \item $D(t,\lambda)$ is compact and convex for $t\in[t_1(\lambda),\infty)_\T$,
 \item $\mathscr{C}(\infty,\lambda):=\displaystyle\lim_{t\rightarrow\infty}\mathscr{C}(t,\lambda)$ exists,
 \item the equality \begin{equation}\label{bep3.36} \bigcap_{t\in[t_1,\infty)_\T} \left[D(t,\lambda)-\mathscr{C}(t,\lambda)\right] = D(\infty,\lambda)-\mathscr{C}(\infty,\lambda)\end{equation} holds, where $D(\infty,\lambda):=\cap_{t\in[t_1,\infty)_\T}D(t,\lambda)$,
 \item $\mathscr{C}(\infty,\lambda)\in D(\infty,\lambda)$.
\end{enumerate}
\end{theorem}

\begin{proof}
In \eqref{ms3.1}, multiply on the left by $(I_n|l^*)$ and on the right by $\left(\begin{smallmatrix} I_n\\l\end{smallmatrix}\right)$ to obtain
\begin{eqnarray*}
 &\left(\widehat{\theta}(t)+\widehat{\phi}(t)l\right)^*\mathscr{U}_{2n}J\left(\widehat{\theta}(t)+\widehat{\phi}(t)l\right) = \left(\widehat{\theta}(t_0)+\widehat{\phi}(t_0)l\right)^*\mathscr{U}_{2n}J\left(\widehat{\theta}(t_0)+\widehat{\phi}(t_0)l\right) \\
 & +2\displaystyle\int_{t_0}^{t}(\theta(s)+\phi(s)l)^*W(s,\lambda)(\theta(s)+\phi(s)l)\Delta s.
\end{eqnarray*}
If $l\in D(t,\lambda)$, then
\begin{equation}\label{ms3.9}
 \left(\widehat{\theta}(t_0)+\widehat{\phi}(t_0)l\right)^*\mathscr{U}_{2n}J\left(\widehat{\theta}(t_0)+\widehat{\phi}(t_0)l\right) 
 +2\displaystyle\int_{t_0}^{t}(\theta(s)+\phi(s)l)^*W(s,\lambda)(\theta(s)+\phi(s)l)\Delta s \le 0;
\end{equation}
from \eqref{ms3.5} we have for $\tau\in[t_0,t)_\T$ that
$$ \int_{t_0}^{\tau}(\theta(s)+\phi(s)l)^*W(s,\lambda)(\theta(s)+\phi(s)l)\Delta s \le \int_{t_0}^{t}(\theta(s)+\phi(s)l)^*W(s,\lambda)(\theta(s)+\phi(s)l)\Delta s \le 0, $$
putting $l\in D(\tau,\lambda)$. Thus (i) holds. For the rest of the proof, see \cite[Theorem 3.6]{b2} and \cite[Theorem 3.2]{ms} using the representation \eqref{ms3.8} of $D(t,\lambda)$.
\end{proof}

By the nesting property above (Theorem \ref{msthm3.2}(v)), there exists a limiting set $D(\infty,\lambda)$ that may consist of a single point. If $l\in D(\infty,\lambda)$, then from \eqref{ms3.9} we see for $t\in[t_0,\infty)_\T$ that
$$ \int_{t_0}^{t}(\theta(s)+\phi(s)l)^*W(s,\lambda)(\theta(s)+\phi(s)l)\Delta s \le -\frac{1}{2} \left(\widehat{\theta}(t_0)+\widehat{\phi}(t_0)l\right)^*\mathscr{U}_{2n}J\left(\widehat{\theta}(t_0)+\widehat{\phi}(t_0)l\right). $$
Consequently the infinite integral satisfies
$$  \int_{t_0}^{\infty}(\theta(s)+\phi(s)l)^*W(s,\lambda)(\theta(s)+\phi(s)l)\Delta s<\infty, $$
ergo the function $\psi(\lambda):=\theta(\lambda)+\phi(\lambda)l$ is $W(\lambda)-$square integrable. It then follows from \eqref{ms3.5} that $\psi(\lambda)$ is $\widetilde{A}-$square integrable, where $\widetilde{A}(t):=\mathscr{U}_{2n}A(t)\mathscr{U}_{2n}^*$.


\begin{remark}
In the \eqref{ms3.7} representation of the set $D(t,\lambda)$ the matrix $\mathscr{C}(t,\lambda)$ can be thought of playing the part of the center, and $\mathscr{R}(t,\lambda)$ that of the radius of the Weyl circles.
\end{remark}

Now consider the adjoint equation \eqref{ms2.4}. Analogous to \eqref{ms3.1}, the fundamental system $\widehat{Z}$ satisfies
\begin{equation}\label{ms3.10}
 \left(\widehat{Z}^*\mathscr{U}_{2n}^{-1}J\widehat{Z}\right)(t) - \left(\widehat{Z}^*\mathscr{U}_{2n}^{-1}J\widehat{Z}\right)(t_0) = 2 \int_{t_0}^{t} Z^*(s)\widetilde{W}(s,\lambda)Z(s)\Delta s,
\end{equation}
where
$$ \widetilde{W}(t,\lambda):=\re\left[(\lambda A(t)+B(t))\mathscr{U}_{2n}^{-1*}\right] = \re\left[\mathscr{U}_{2n}^{-1}\left(\overline{\lambda} A(t)+B^*(t)\right)\right], $$
and $W(t,\lambda)=\mathscr{U}_{2n}\widetilde{W}(t,\lambda)\mathscr{U}_{2n}^*$ for $W$ given in \eqref{Wdef}.


\begin{definition}
Let $\lambda_0\in\C$, and let $\mathscr{U}_{2n}$ be given as above in \eqref{Udef}. Then $(\lambda_0,\mathscr{U}_{2n}^{-1})$ is called an admissible pair for the adjoint equation \eqref{ms2.4} if and only if
$$ \widetilde{W}(t,\lambda)=\re\left[(\lambda A(t)+B(t))\mathscr{U}_{2n}^{-1*}\right] \ge 0, \quad t\in[t_0,\infty)_\T, $$
and define the set
\begin{eqnarray}
 \widetilde{\Lambda}(\lambda_0,\mathscr{U}_{2n}^{-1}) &:=&
 \left\{\lambda\in\C:\;\exists\;\delta>0\;\ni\;\re[(\lambda-\lambda_0)A(t)\mathscr{U}_{2n}^{-1*}] \ge \delta \mathscr{U}_{2n}^{-1} A(t) \mathscr{U}_{2n}^{-1*}\; \forall\; t\in[t_0,\infty)_\T \right\} \nonumber \\
  &=& \left\{\lambda\in\C:\;\exists\;\delta>0\;\ni\;\re[(\lambda-\lambda_0)\mathscr{U}_{2n}A(t)] \ge \delta A(t)\; \forall\; t\in[t_0,\infty)_\T \right\}. \label{ms3.11}
\end{eqnarray}
\end{definition}

Similar to the authors in \cite[p. 89]{ms}, we make the following remarks for the sake of completeness.


\begin{remark}
(i) Note that $(\lambda_0,\mathscr{U}_{2n})$ is an admissible pair for \eqref{maineq} if and only if $(\lambda_0,\mathscr{U}_{2n}^{-1})$ is an admissible pair for the adjoint equation \eqref{ms2.4}. For $\lambda\in \Lambda(\lambda_0,\mathscr{U}_{2n}) \cap \widetilde{\Lambda}(\lambda_0,\mathscr{U}_{2n}^{-1})$ and sufficiently large $t\in\T$, the Weyl-Sims sets for \eqref{ms2.4} are 
$$ \widetilde{D}(t,\lambda)=\left\{l^*\in\C^{n,n}: l\in D(t,\lambda)\right\}=D^*(t,\lambda); $$
see the discussion in \cite{b2}.

(ii) For $\lambda\in\widetilde{\Lambda}(\lambda_0,\mathscr{U}_{2n}^{-1})$ we have from \eqref{ms3.10} that $l^*\in\widetilde{D}(t,\lambda)$ if and only if
$$ \int_{t_0}^{t}(\eta(s)+\chi(s)l^*)^*\widetilde{W}(s,\lambda)(\eta(s)+\chi(s)l^*)\Delta s \le -\frac{1}{2} \left(\widehat{\eta}(t_0)+\widehat{\chi}(t_0)l^*\right)^*\mathscr{U}_{2n}^{-1}J\left(\widehat{\eta}(t_0)+\widehat{\chi}(t_0)l^*\right). $$
If $l^*\in\widetilde{D}(\infty,\lambda)$, then $\zeta(\lambda)=\eta(\lambda)+\chi(\lambda)l^*$ satisfies
$$ \int_{t_0}^{\infty} \zeta^*(s)\widetilde{W}(s,\lambda)\zeta(s)\Delta s < \infty, $$
making $\zeta$ a $\widetilde{W}(\lambda)-$square integrable function.

(iii) If $W(t,\lambda)\ge\delta A(t)$ for all $t\in[t_0,\infty)_\T$, some $\delta>0$ and $\lambda\in\C$, then $L_{W(\lambda)}^2(t_0,\infty)_\T\subseteq L_{A}^2(t_0,\infty)_\T$. This condition holds in the following two cases:
\begin{itemize}
 \item if $\lambda\in\widetilde{\Lambda}(\lambda_0,\mathscr{U}_{2n}^{-1})$, as seen by \eqref{ms3.2} and \eqref{ms3.11}, or
 \item if \begin{equation}\label{ms3.12} 
            \lambda\in\Lambda(\lambda_0,\mathscr{U}_{2n}) \quad\text{and}\quad \widetilde{A}(t)\ge \gamma A(t) \quad\text{for some}\quad \gamma>0
          \end{equation}
       by using \eqref{ms3.2} and \eqref{ms3.3}.
\end{itemize}

(iv) If $\lambda\in\Lambda(\lambda_0,\mathscr{U}_{2n})$, then $\widetilde{W}(t,\lambda)\ge\delta A(t)$ for some $\delta>0$, ergo $L_{\widetilde{W}(\lambda)}^2(t_0,\infty)_\T\subseteq L_{A}^2(t_0,\infty)_\T$. Consequently, if $\lambda\in\Lambda(\lambda_0,\mathscr{U}_{2n})\cap\widetilde{\Lambda}(\lambda_0,\mathscr{U}_{2n}^{-1})$, then we have
$$ L_{W(\lambda)}^2(t_0,\infty)_\T \cup L_{\widetilde{W}(\lambda)}^2(t_0,\infty)_\T \subseteq L_{A}^2(t_0,\infty)_\T, $$
and $\psi$ and $\zeta$ are $A-$square integrable.

(v) Using \eqref{ms3.3} and \eqref{ms3.11}, condition \eqref{ms3.12} above yields
\begin{equation}\label{ms3.13}
 \Lambda(\lambda_0,\mathscr{U}_{2n}) \subseteq \widetilde{\Lambda}(\lambda_0,\mathscr{U}_{2n}^{-1})
\end{equation}
for admissible $(\lambda_0,\mathscr{U}_{2n})$. Besides \eqref{ms3.12}, if the reverse inequality $A(t)\ge\tilde{\gamma}\widetilde{A}(t)$ holds for some $\tilde{\gamma}>0$, that is to say if $\widetilde{A}(t)\asymp A(t)$, then we have equality in \eqref{ms3.13}.

(vi) As in \cite{b2} and \cite{ms}, the structure of the shifted limit set $D(\infty,\lambda)-\mathscr{C}(\infty,\lambda)$ gives information about the number of $W(\lambda)-$square integrable solutions to system \eqref{maineq}. To be more explicit, let $\mathscr{N}\mathcal{N}(\lambda):=\bigcup_{N\in D(\infty,\lambda)}\range(N -\mathscr{C}(\infty,\lambda))$ and $r$ be the dimension of the linear hull
of $\mathcal{N}(\lambda)$. Then there are at least $n+r$ linearly independent $W(\lambda)-$square integrable solutions of \eqref{maineq}, and if $\mathscr{R}(t,\lambda)\not\longrightarrow 0$ as $t\rightarrow\infty$ in the time scale, the number is exactly $n+r$.

In addition, if $\lambda\in\Lambda(\lambda_0,\mathscr{U}_{2n})\cap\widetilde{\Lambda}(\lambda_0,\mathscr{U}_{2n}^{-1})$, then \eqref{ms2.4} has precisely $n$ linearly independent $\widetilde{W}(\lambda)-$square integrable solutions if $\mathscr{R}(t,\lambda)\rightarrow 0$ as $t\rightarrow\infty$ in the time scale. If $r=0$, then we also have $\tilde{r}=0$, where $\tilde{r}$ is the corresponding number for the adjoint equation, and at least one of \eqref{maineq}, \eqref{ms2.4} has exactly $n$ linearly independent solutions that are $W(\lambda)$, $\widetilde{W}(\lambda)-$square integrable,
respectively.
\end{remark}


\section{Examples}\label{sec4}

The examples given in the continuous case \cite[Section 3]{b2} require $A$ defined in \eqref{ABdef} to be positive semi-definite. In the discrete case \cite{ms}, however, the blanket positive definite assumption $A>0$ rules out giving the corresponding examples for difference equations. Thus we are able in this section to extend the examples given in \cite[Section 3]{b2} to difference equations and other cases on Sturmian time scales, and mention a new result for general even-order dynamic equations.


\begin{example}
For $n=1$, set
\begin{equation}\label{bep3.22}
 \mathscr{U}_2=\begin{pmatrix} -u & 0 \\ 0 & \bar{u} \end{pmatrix}
\end{equation}
for some nonzero $u\in\C$. Then \eqref{ms3.6} yields
$$ 2\begin{pmatrix} S(t) & T(t) \\ T^*(t) & P(t) \end{pmatrix} = \begin{pmatrix} \theta_1(t) & \phi_1(t) \\ \theta_2^\rho(t) & \phi_2^\rho(t) \end{pmatrix}^* \begin{pmatrix} 0 & u(t) \\ \bar{u}(t) & 0 \end{pmatrix} \begin{pmatrix} \theta_1(t) & \phi_1(t) \\ \theta_2^\rho(t) & \phi_2^\rho(t) \end{pmatrix}, $$
so that when $P=P(t,\lambda)>0$ we have, also using \eqref{circrad},
$$ P = \re\left(u\overline{\phi}_1\phi_2^\rho\right), \quad \mathscr{R} = \frac{1}{4P}|u|^2|\Gamma|^2, \quad \mathscr{C} = \frac{-1}{2P}\left(u\overline{\phi}_1\theta_2^\rho + \overline{u}\theta_1\overline{\phi}_2^\rho\right), $$
where $\Gamma=\theta_1\phi_2^\rho-\theta_2^\rho\phi_1$. Thus \eqref{ms3.7} is
$$ D(t,\lambda) = \left\{l\in\C: |l-\mathscr{C}(t,\lambda)| \le \frac{|u||\Gamma(t,\lambda)|}{2P(t,\lambda)}\right\}. $$
As a special case of this, consider on Sturmian time scales the second-order scalar Sturm-Liouville problem
\begin{equation}\label{bep3.23}
 -(pv^\nabla)^\Delta(t) + q(t) v(t) = \lambda w(t) v(t), \quad t\in[t_0,\infty)_\T,
\end{equation}
where $p$ and $q$ are complex-valued functions, with $p^{-1},q,w\in L^1_{\loc}[t_0,\infty)_\T$ such that $p\neq 0$ and $w>0$ on $[t_0,\infty)_\T$. Then \eqref{bep3.23} can be written in the form \eqref{maineq}
$$ y=\begin{pmatrix} v \\ p^\sigma v^\Delta \end{pmatrix}, \quad \widehat{y}=\begin{pmatrix} v \\ p v^\nabla \end{pmatrix}, \quad A=\begin{pmatrix} w & 0 \\ 0 & 0 \end{pmatrix}, \quad B=\begin{pmatrix} -q & 0 \\ 0 & 1/p^\sigma \end{pmatrix}; $$
note that $A$ and $B$ satisfy \eqref{ABdef} and \eqref{msA1}, respectively. If we choose $u=e^{i\eta}$ for some $\eta\in\R$ in \eqref{bep3.23}, we see from \eqref{Wdef} that
\begin{equation}\label{bep3.24}
  W(t,\lambda)=\re\left[\mathscr{U}_{2}(\lambda A(t)+B(t))\right]=\begin{pmatrix} \re[e^{i\eta}(q-\lambda w)(t)] & 0 \\ 0 & \displaystyle\frac{1}{|p^\sigma(t)|^2}\re[e^{i\eta}p^\sigma(t)] \end{pmatrix}. 
\end{equation}
Thus for this choice of $\mathscr{U}_2$ we have
$$ (\lambda_0,\mathscr{U}_{2})\in\mathscr{S} \iff \re[e^{i\eta}(q-\lambda w)(t)], \; \displaystyle\frac{1}{|p^\sigma(t)|^2}\re[e^{i\eta}p^\sigma(t)]\ge 0 \quad \forall t\in[t_0,\infty)_\T. $$
In addition we see that for $(\lambda_0,\eta)$ such that $(\lambda_0,\mathscr{U}_{2})\in\mathscr{S}$,
$$  \Lambda(\lambda_0,\mathscr{U}_{2})=\left\{\lambda\in\C:  \re[(\lambda-\lambda_0)e^{i\eta}]<0\right\}, $$
just as in \cite[(3.25)]{b2}. It is easy to show that the definiteness condition \eqref{bep3.14} holds, see \cite[p. 425]{b2}.
\end{example}


\begin{example}\label{example2}
Consider the fourth-order scalar problem \cite[(7.1)]{agh}
\begin{equation}\label{bep3.26}
  (p_2v^{\Delta\nabla})^{\nabla\Delta}(t)-(p_1v^\nabla)^\Delta(t) + p_0(t) v(t) = \lambda w(t) v(t), \quad t\in[t_0,\infty)_\T,
\end{equation}
where $p_0$, $p_1$, and $p_2$ are complex-valued functions with $p_2\neq 0$ on $[t_0,\infty)_\T$, and $w\in L^2_{\loc}[t_0,\infty)_\T$ satisfies $w>0$ on $[t_0,\infty)_\T$. For the quasi-derivatives given by
$$ v^{[1]}=v^\Delta, \quad v^{[2]}=p_2^{\sigma} v^{\Delta\Delta}, \quad v^{[3]}=p_1^{\sigma}v^\Delta-(v^{[2]\rho})^\Delta, $$
we introduce the vector
$$ y = \left(v, v^{[1]}, v^{[3]}, v^{[2]} \right)^{\trans}, \quad\text{with}\quad \widehat{y} = \left(v, v^{[1]}, v^{[3]\rho}, v^{[2]\rho} \right)^{\trans}. $$
Then \eqref{bep3.26} can be written in the form of \eqref{maineq} if we take
$$ A=\diag\{w,0,0,0\} \quad\text{and}\quad B=\begin{pmatrix} -p_0 &  &  &  \\  & -p_1^\sigma & 1 &  \\  & 1 & 0 &  \\  &  &  & 1/p_2^\sigma \end{pmatrix}, $$
with the unstated entries being zero. Note that $A$ and $B$ satisfy \eqref{ABdef} and \eqref{msA1}, respectively. For some $\eta\in\R$ choose
\begin{equation}\label{u4}
 \mathscr{U}_4:=\begin{pmatrix} -e^{i\eta}I_2 & 0_2 \\ 0_2 & e^{-i\eta}I_2 \end{pmatrix}; 
\end{equation}
it follows from \eqref{Wdef} and \eqref{realpart} that
$$ W(t,\lambda) = \diag\left\{\re\left[e^{i\eta}(p_0(t)-\lambda w(t))\right], \re\left[e^{i\eta}p_1^\sigma(t)\right], 0,  \re\left[e^{i\eta}p_2^\sigma(t)\right]/|p_2^\sigma(t)|^2 \right\}, $$
so that for this choice of $\mathscr{U}_4$ we have
$$ (\lambda_0,\mathscr{U}_{4})\in\mathscr{S} \iff \re\left[e^{i\eta}(p_0(t)-\lambda_0 w(t))\right], \re\left[e^{i\eta}p_1^\sigma(t)\right], \re\left[e^{i\eta}p_2^\sigma(t)\right]\ge 0 \quad \forall t\in[t_0,\infty)_\T. $$
In addition we see that for $\lambda_0$ such that $(\lambda_0,\mathscr{U}_{4})\in\mathscr{S}$,
\begin{equation}\label{bep3.25}
  \Lambda(\lambda_0,\mathscr{U}_{4})=\left\{\lambda\in\C:  \re\left[(\lambda-\lambda_0)e^{i\eta}\right]<0\right\}.
\end{equation}
As in the previous example, the definiteness condition \eqref{bep3.14} holds.
\end{example}


\begin{example}
Using Example $\ref{example2}$ as a guide, consider on $[t_0,\infty)_\T$ the formally self-adjoint $2n$th-order dynamic equation \cite{agh} of the form (suppressing the independent variable)
\begin{equation}\label{exeqeven}
 (-1)^n(p_nv^{\Delta^{n-1}\nabla})^{\nabla^{n-1}\Delta} + \dots  - (p_{3}v^{\Delta^2\nabla})^{\nabla^2\Delta} + (p_{2}v^{\Delta\nabla})^{\nabla\Delta} - (p_{1}v^\nabla)^\Delta + p_0v = \lambda w v,
\end{equation}
where $p_j$ is a complex-valued function for $j=0,1,\ldots,n-1$ with $p_n\neq 0$ on $[t_0,\infty)_\T$, and $w\in L^2_{\loc}[t_0,\infty)_\T$ satisfies $w>0$ on $[t_0,\infty)_\T$. Let $A=\diag\{w,0,\ldots,0\}$, and let $B$ be as in \eqref{ABdef}, where we take
$$ B_1 = \diag\{-p_0,-p_1^\sigma,-p_2^\sigma,\ldots,-p_{n-1}^\sigma\}, \quad B_2 = \subdiag\{1,1,\ldots,1\}, $$
$$ B_3 = \supdiag\{1,1,\ldots,1\}=B_2^{\trans}, \quad B_4 = \diag\{0,\ldots,0,1/p_{n}^\sigma\}. $$
Here $\subdiag$ means the matrix with all zero entries except on the subdiagonal; similarly $\supdiag$ has nonzero entries only on the superdiagonal. Clearly the conditions in \eqref{ABdef} and \eqref{msA1} are satisfied. Set 
$$ y = \left(v, v^{[1]}, v^{[2]}, \ldots, v^{[n-1]}, v^{[2n-1]}, v^{[2n-2]}, \ldots, v^{[n]} \right)^{\trans}, $$
where the quasi-derivatives are given by
\begin{eqnarray*}
  v^{[k]} &=& v^{\Delta^k}, \quad 1\le k\le n-1,   \\
  v^{[n]} &=& p_n^\sigma v^{\Delta^{n}},  \\
  v^{[n+k]} &=& p_{n-k}^\sigma v^{[n-k]} - \left(v^{[n+k-1]\rho}\right)^\Delta, \quad 1 \le k \le n-1.
\end{eqnarray*}
Then \eqref{exeqeven} can be written in the form of \eqref{maineq}. For some $\eta\in\R$ choose
$$ \mathscr{U}_{2n}:=\begin{pmatrix} -e^{i\eta}I_n & 0_n \\ 0_n & e^{-i\eta}I_n \end{pmatrix}. $$
It follows from \eqref{Wdef} and \eqref{realpart} that
\begin{eqnarray*}
 W(t,\lambda) = \diag\left\{\re\left[e^{i\eta}(p_0(t)-\lambda w(t))\right], \re\left[e^{i\eta}p_1^\sigma(t)\right], \re\left[e^{i\eta}p_2^\sigma(t)\right], \ldots, \re\left[e^{i\eta}p_{n-1}^\sigma(t)\right], \right. \\
 \left. 0, \ldots, 0, \re\left[e^{i\eta}p_n^\sigma(t)\right]/|p_n^\sigma(t)|^2 \right\}, 
\end{eqnarray*}
so that for this choice of $\mathscr{U}_{2n}$ we have
$$ (\lambda_0,\mathscr{U}_{2n})\in\mathscr{S} \iff \re\left[e^{i\eta}(p_0(t)-\lambda_0 w(t))\right], \re\left[e^{i\eta}p_1^\sigma(t)\right], \ldots, \re\left[e^{i\eta}p_n^\sigma(t)\right]\ge 0 $$
for all $t\in[t_0,\infty)_\T$.
In addition we see that for $\lambda_0$ such that $(\lambda_0,\mathscr{U}_{2n})\in\mathscr{S}$,
$$ \Lambda(\lambda_0,\mathscr{U}_{2n})=\left\{\lambda\in\C:  \re\left[(\lambda-\lambda_0)e^{i\eta}\right]<0\right\}. $$
As in the previous examples, the definiteness condition \eqref{bep3.14} holds.

This example is completely new, as general even-order equations are not mentioned in either \cite{b2} or \cite{ms}; clearly the special cases $\T=\R$ and $\T=\Z$ are included here.
\end{example}


\begin{example}
Consider the Orr-Sommerfeld equation on time scales given by
\begin{equation}\label{bep3.27}
 (-D^2+a^2)^2u+iaR\left[V(-D^2+a^2)u+uD^2V\right]=\lambda (-D^2+a^2)u, \quad D^2u\equiv u^{\nabla\Delta}
\end{equation}
on some interval $I\subseteq [t_0,\infty)_\T$, where $a>0$ is the wave number, $R>0$ is the Reynolds number, and $V$ is a real-valued flow velocity profile perpendicular to $I$; see \cite[Example 3.4]{b2} and Orszag \cite{orszag}. If we introduce the variables
$$ y_1 = -u^{\nabla\Delta}+a^2u, \quad y_2 = u, \quad y_3 = (-u^{\nabla\Delta}+a^2u)^\Delta, \quad y_4 = u^\Delta, $$
then 
$$ \widehat{y}(t) = \left(-u^{\nabla\Delta}+a^2u, u, (-u^{\nabla\Delta}+a^2u)^\nabla, u^\nabla \right)^{\trans}(t), $$
and we see that \eqref{bep3.27} is equivalent to the Hamiltonian system \eqref{maineq} with
$$ A=\diag\{1,0,0,0\} \quad\text{and}\quad B=\begin{pmatrix} -a^2-iaRV & -iaRV^{\nabla\Delta} & 0 & 0 \\ 1 & -a^2 & 0 & 0 \\ 0 & 0 & 1 & 0 \\ 0 & 0 & 0 & 1 \end{pmatrix} $$
such that \eqref{ABdef} and \eqref{msA1} are easily satisfied. We will choose the same matrix $\mathscr{U}_4$ as in \eqref{u4} for some $\eta\in\R$; then by \eqref{Wdef} we have
$$ W(t,\lambda) = \begin{pmatrix} a^2\cos\eta-aRV\sin\eta-\re[\lambda e^{i\eta}] & \frac{1}{2}(aRV^{\nabla\Delta}ie^{i\eta}-e^{-i\eta}) & 0 & 0 \\ -\frac{1}{2}(aRV^{\nabla\Delta}ie^{-i\eta}+e^{i\eta}) & a^2\cos\eta & 0 & 0 \\ 0 & 0 & \cos\eta & 0 \\ 0 & 0 & 0 & \cos\eta \end{pmatrix}. $$
As in the case $\T=\R$ \cite{b2}, we thus have
\begin{equation}\label{bep3.28}
 W(t,\lambda)\ge 0\iff \begin{cases} \cos\eta>0 &  \text{and} \\ \re(\lambda e^{i\eta}) \le & a^2\cos\eta-aRV(t)\sin\eta \\
 &  - \displaystyle \frac{1+\left(aRV^{\nabla\Delta}(t)\right)^2+2aRV^{\nabla\Delta}(t)\sin(2\eta)}{4a^2\cos\eta}, \end{cases}
\end{equation}
so that $W(t,\lambda)>0$ if the last inequality in \eqref{bep3.28} is strict. Consequently, $(\lambda_0,\mathscr{U}_4)\in\mathscr{S}$ if and only if the right-hand side of \eqref{bep3.28} holds for $\lambda_0$ in place of $\lambda$, for all $t\in I$. We can also show that for such an $\eta$ and $\lambda_0$ we get that $ \Lambda(\lambda_0,\mathscr{U}_{4})$ is again as in \eqref{bep3.25}. Similarly, the definiteness condition \eqref{bep3.14} holds.
\end{example}


\section{Definition of the operators $L_\xi$ and $\widetilde{L}_\xi$}\label{sec5}

We will show below that for fixed $\xi\in\Lambda(\lambda_0,\mathscr{U}_{2n})$ and $M_0\in D(\infty,\xi)$ there exists for $\lambda\in\Lambda(\lambda_0,\mathscr{U}_{2n})$ a matrix-valued function $M(\lambda)$ such that $M(\xi)=M_0$. In addition, the Weyl solutions satisfy condition \eqref{ms4.3} at infinity. This in turn will allow us to introduce the operator $L_\xi$ associated with \eqref{maineq} and the operator $\widetilde{L}_\xi$ associated with the adjoint system \eqref{ms2.4}.  Throughout the discussion we will assume that
\begin{equation}\label{ms4.1}
 \widetilde{A}(t)=\mathscr{U}_{2n}A(t)\mathscr{U}_{2n}^* \asymp A(t), \quad t\in[t_0,\infty)_\T.
\end{equation}
Indeed this holds if, for example, $A_k(t)=a_k(t)I_n+\widetilde{A}_k(t)$ for arbitrary $a_k(t)$ and $1/c\le\widetilde{A}_k(t)\le c$ for some real constant $c>1$, for all $t\in[t_0,\infty)_\T$ and $k=1,2$.


\begin{theorem}\label{thm4.1}
Let $\xi\in\Lambda(\lambda_0,\mathscr{U}_{2n})$ for $(\lambda_0,\mathscr{U}_{2n})\in\mathscr{S}$, and $M_0\in D(\infty,\xi)$ be fixed. Then there exists a function $M:\Lambda(\lambda_0,\mathscr{U}_{2n})\rightarrow\C^{n,n}$ such that $M(\xi)=M_0$, $M(\lambda)\in D(\infty,\lambda)$ and 
\begin{equation}\label{ms4.2}
 M(\lambda) - M_0 = (\lambda-\xi)\int_{t_0}^{\infty} \zeta^*(t,\xi)A(t)\psi(t,\lambda) \Delta t = (\lambda-\xi)\int_{t_0}^{\infty} \zeta^*(t,\lambda)A(t)\psi(t,\xi) \Delta t
\end{equation}
for $\lambda\in\Lambda(\lambda_0,\mathscr{U}_{2n})$, where $\psi(t,\lambda):=\theta(t,\lambda)+\phi(t,\lambda)M(\lambda)$ and $\zeta(t,\lambda):=\eta(t,\lambda)+\chi(t,\lambda)M(\lambda)^*$. In addition, 
\begin{equation}\label{ms4.3}
 \lim_{t\rightarrow\infty}\widehat{\zeta}^*(t,\xi)J\widehat{\psi}(t,\lambda) = \lim_{t\rightarrow\infty}\widehat{\zeta}^*(t,\lambda)J\widehat{\psi}(t,\xi) = 0.
\end{equation}
\end{theorem}

\begin{proof}
See \cite[Theorem 5.4]{b2} and \cite[Theorem 4.1]{ms}, as the proof is unchanged from the continuous and discrete cases except for minor notational modifications.
\end{proof}


\begin{remark}
Clearly, $M(\lambda)$ has only one possible value if $D(\infty,\lambda)$ has only one element. Consequently, if there is at least one point $\xi\in\Lambda(\lambda_0,\mathscr{U}_{2n})$ such that $D(\infty,\lambda)$ has only one element, the function $M$ is uniquely determined.
\end{remark}

From this point on we assume that $(\lambda_0,\mathscr{U}_{2n})\in\mathscr{S}$ is fixed, and that the function $M$ and the selected value $\xi$ are as in Theorem \ref{thm4.1} above. Corresponding to system \eqref{maineq} is the inhomogeneous system
\begin{equation}\label{ms5.1}
 J\widehat{y}^\Delta(t)=\Big(\lambda A(t)+B(t)\Big) y(t) + A(t)f(t), \quad t\in[t_0,\infty)_\T,
\end{equation}
or in its expanded form (suppressing the $t$),
\begin{equation*}
 \begin{cases} 
   -&(y_2^\rho)^\Delta = \big(\lambda A_1+B_1\big)y_1 + B_2y_2+A_1f_1, \\
   &y_1^\Delta  = B_3y_1 + (\lambda A_2+B_4)y_2+A_2f_2. \end{cases}
\end{equation*}
Note that if $\mu(t)\neq 0$, then
$$ y_2(t)=\left(I+\mu(t)B_2(t)\right)^{-1}\left[y_2^\rho(t)-\mu(t)(\lambda A_1(t)+B_1(t))y_1(t)-\mu(t)A_1(t)f_1(t)\right]. $$
A solution $y$ of system \eqref{ms5.1} is thus related to $\widehat{y}$ (see \eqref{leftshift}) using
\begin{equation}\label{ms5.2}
  y(t) = H(t)\widehat{y}(t)+N(t)A(t)f(t), \quad t\in[t_0,\infty)_\T,
\end{equation}
where $H$ is given in \eqref{ms2.2} and
\begin{equation}\label{Ndef}
 N(t):=\begin{pmatrix} 0_n & 0_n \\ -\mu(t)E_2(t) & 0_n \end{pmatrix}=\begin{pmatrix} 0_n & 0_n \\ -\mu(t)\Big(I_n+\mu(t)B_2(t)\Big)^{-1} & 0_n \end{pmatrix}. 
\end{equation}
As $(I_n-\mu B_2E_2)=E_2$ we have $(\lambda A+B)N+I_{2n}=\widetilde{H}^*$ from \eqref{Htilde}, so that \eqref{ms5.1} is equivalent to
\begin{equation}\label{ms5.3}
 J\widehat{y}^\Delta(t) = \Big(\lambda A(t)+B(t)\Big)H(t)\widehat{y}(t) + \widetilde{H}^*(t)A(t)f(t), \quad t\in[t_0,\infty)_\T.
\end{equation}
In a similar way, for the adjoint problem to \eqref{ms5.1}, namely
\begin{equation}\label{ms5.4}
 J\widehat{z}^\Delta(t) = \Big(\overline{\lambda}A(t)+B^*(t)\Big) z(t) + A(t)f(t), \quad t\in[t_0,\infty)_\T,
\end{equation}
we have $z_1=E_2^*z_1^\sigma-\mu E_2^*(\overline{\lambda}A_2+B^*_4)z_2-\mu E_2^*A_2f_2$ and thus the relation
\begin{equation}\label{zzwide}
 z(t)=\widetilde{H}(t)\widehat{z}^{\;\sigma}(t)+N^*(t)A(t)f(t), \quad t\in[t_0,\infty)_\T. 
\end{equation} 
Moreover, since $(\overline{\lambda}A+B^*)N^*+I_{2n}=H^*$ for $H$ in \eqref{ms2.2}, we see that \eqref{ms5.4} is equivalent to
$$ J\widehat{z}^{\;\Delta}(t) = H^*(t)\Big(\overline{\lambda}A(t)+B^*(t)\Big)\widehat{z}^{\;\sigma}(t) + H^*(t)A(t)f(t) $$
as in \eqref{ms2.6}, for $t\in[t_0,\infty)_\T$. Define
\begin{eqnarray*}
 &G(t,s,\lambda)&:=\begin{cases} \psi(t,\lambda)\chi^*(s,\lambda) &: t_0\le s<t<\infty, \\ 
                               \phi(t,\lambda)\zeta^*(s,\lambda)+N(t)\delta_{ts} &: t_0\le t\le s <\infty, \end{cases} \\
 &\widetilde{G}(t,s,\lambda)&:=\begin{cases} \chi(t,\lambda)\psi^*(s,\lambda) &: t_0\le t<s<\infty, \\ 
                               \zeta(t,\lambda)\phi^*(s,\lambda)+N^*(t)\delta_{ts} &: t_0\le s\le t <\infty, \end{cases} \\
 &&     = G^*(s,t,\lambda),
\end{eqnarray*}
where $\delta_{ts}$ is the Dirac delta function, i.e., the function that satisfies the sifting property
$$ \int_{t_0}^{\infty}f(s)\delta_{ts}\Delta s=f(t). $$
In a subsequent lemma we prove that $G$ and $\widetilde{G}$ are Green's matrices for \eqref{ms5.1} and \eqref{ms5.4}, respectively. For $f\in L_A^2(t_0,\infty)_\T$, define
\begin{eqnarray}
 && (R_{\lambda}f)(t):=\int_{t_0}^{\infty} G(t,s,\lambda)A(s)f(s)\Delta s, \label{bep5.6} \\
 && \left(\widetilde{R}_{\lambda}f\right)(t):=\int_{t_0}^{\infty} \widetilde{G}(t,s,\lambda)A(s)f(s)\Delta s.\label{bep5.7}
\end{eqnarray}
We assume that $R_{\lambda}$ and $\widetilde{R}_{\lambda}$ defined in \eqref{bep5.6} and \eqref{bep5.7} are one-to-one as operators from $L_A^2(t_0,\infty)_\T$ into itself, in other words,
\begin{eqnarray}
 && f\in L_A^2(t_0,\infty)_\T,\; A(R_{\xi}f)(t)=0\; \forall t\in(t_0,\infty)_\T\implies (Af)(t)=0 \; \forall t\in(t_0,\infty)_\T,  \label{bep5.42} \\
 && g\in L_A^2(t_0,\infty)_\T,\; A(\widetilde{R}_{\xi}g)(t)=0\; \forall t\in(t_0,\infty)_\T\implies (Ag)(t)=0 \; \forall t\in(t_0,\infty)_\T. \label{bep5.43} \\ 
\end{eqnarray}
As pointed out in \cite[p. 444]{b2}, it will become evident in the sequel that these same conditions hold when $\xi$ is replaced by any $\lambda\in\Lambda(\lambda_0,\mathscr{U}_{2n})$. We now define the following operators $L_{\xi}$ and $\widetilde{L}_{\xi}$ in a natural way. Set
\begin{eqnarray}
 \mathscr{D}(L_{\xi})&:=&\left\{y\in L_A^2(t_0,\infty)_\T: y\in \acloc[t_0,\infty)_\T, \right. \nonumber \\
 & & (ly)(t):= \begin{cases} J\widehat{y}^{\Delta}(t)-B(t)y(t)=A(t)f(t) &: t\in[\sigma(t_0),\infty)_\T, \\ \left(\begin{smallmatrix}-y_2(t_0)/\mu(t_0) \\ y_1^\Delta(t_0) \end{smallmatrix}\right)-B(t_0)y(t_0) = A(t_0)f(t_0) &: t=t_0>\rho(t_0), \end{cases}  \nonumber \\ 
 & & \left. \text{for some}\;f\in L_A^2(t_0,\infty)_\T,\;\text{ and }\;\lim_{t\rightarrow\infty}\widehat{\zeta}^*(t,\xi)J\widehat{y}(t)=0\right\}, \label{bep5.44}\\
 L_{\xi}y&:=&f\;\; \text{for}\; y\in \mathscr{D}(L_{\xi})\;\text{satisfying}\; ly=Af\;\text{on}\;[t_0,\infty)_\T, \label{Lxi}
\end{eqnarray}
where the second line in the definition of $(ly)$ holds only if $t_0$ is a left-scattered (and thus right-scattered) point, and if we take $y_2^\rho(t_0)=0$. Similarly, set
\begin{eqnarray}
 \mathscr{D}\left(\widetilde{L}_{\xi}\right)&:=&\left\{z\in L_A^2(t_0,\infty)_\T: z\in \acloc[t_0,\infty)_\T, \right. \nonumber \\
 & & (\widetilde{l}z)(t):= \begin{cases} J\widehat{z}^{\Delta}(t)-B^*(t)z(t)=A(t)g(t) &: t\in[\sigma(t_0),\infty)_\T, \\  \left(\begin{smallmatrix}-z_2(t_0)/\mu(t_0) \\ z_1^\Delta(t_0) \end{smallmatrix}\right)-B^*(t_0)z(t_0) = A(t_0)g(t_0) &: t=t_0>\rho(t_0), \end{cases} \nonumber \\
 & & \left. \text{for some}\;g\in L_A^2(t_0,\infty)_\T,\;\text{ and }\; \lim_{t\rightarrow\infty}\widehat{\psi}^*(t,\xi)J\widehat{z}(t)=0\right\}, \label{bep5.45}\\
 \widetilde{L}_{\xi}z&:=&g\;\; \text{for}\; z\in \mathscr{D}\left(\widetilde{L}_{\xi}\right)\;\text{satisfying}\; \widetilde{l}z=Ag\;\text{on}\;[t_0,\infty)_\T, \label{Ltildexi}
\end{eqnarray}
where the second line in the definition of $(\widetilde{l}y)$ holds only if $t_0$ is a left-scattered point, and if we take $z_2^\rho(t_0)=0$.
We remark that $\mathscr{D}(L_{\xi})$ and $\mathscr{D}\left(\widetilde{L}_{\xi}\right)$ consist of all equivalence classes in $L_A^2(t_0,\infty)_\T$ such that at least one representative of the class satisfies the conditions in the definitions of $\mathscr{D}(L_{\xi})$ and $\mathscr{D}\left(\widetilde{L}_{\xi}\right)$, respectively. In particular, using \eqref{bep5.42} and \eqref{bep5.43} we see that this representative is always unique, and thus $L_\xi$ and $\widetilde{L}_\xi$ are well defined. For more details, see the discussion in \cite[p. 445]{b2}.


\section{The Resolvent Sets}\label{sec6}

In this section we analyze the operators $L_{\xi}$ and $\widetilde{L}_{\xi}$ defined in the previous section in \eqref{Lxi} and \eqref{Ltildexi}, respectively, and establish their resolvents, which turn out to be $R_{\lambda}$ and $\widetilde{R}_{\lambda}$ from \eqref{bep5.6} and \eqref{bep5.7}, respectively. This operator $R_{\lambda}$ will have inverse operator properties relative to $L_{\xi}-\xi$, in particular that $L_{\xi}$ has resolvent set $\Lambda(\lambda_0,\mathscr{U}_{2n})$ with resolvent operator $R_{\lambda}$. In addition, we prove that $L_{\xi}$ and $\widetilde{L}_{\xi}$ are adjoints. 


\begin{lemma}\label{lemma5.1}
Let $\lambda\in\Lambda(\lambda_0,\mathscr{U}_{2n})\cap\widetilde{\Lambda}(\lambda_0,\mathscr{U}^{-1}_{2n})$ and $f\in L_A^2(t_0,\infty)_\T$. Then $R_{\lambda}f$ is a solution of \eqref{ms5.1} and satisfies $(R_{\lambda}f)_2^\rho(t_0)=0_n$. In particular, it satisfies the boundary conditions $\widehat{\chi}^{\;*}(t_0)J\left(\widehat{R}_{\lambda}f\right)(t_0)=0$ and $\lim_{t\rightarrow\infty}\widehat{\zeta}^{\;*}(t,\lambda)J\left(\widehat{R}_{\lambda}f\right)(t)=0$, where
\begin{equation}\label{ms5.5}
 \left(\widehat{R}_{\lambda}f\right)(t)=\widehat{\psi}(t,\lambda)\int_{t_0}^{t}\chi^*(s,\lambda)A(s)f(s)\Delta s + \widehat{\phi}(t,\lambda)\int_{t}^{\infty}\zeta^*(s,\lambda)A(s)f(s)\Delta s, \quad t\in[t_0,\infty)_\T.
\end{equation}
\end{lemma}

\begin{proof}
We use the method of variation of parameters. Let $\widehat{U}\in\C^{2n,2n}$ have the form $\widehat{U}=\widehat{Y}\mathscr{M}$, where $\mathscr{M}\in\C^{2n,2n}$ is to be determined and $Y=(\theta|\phi)$ is a fundamental matrix solution for the corresponding homogeneous system \eqref{ms2.3} satisfying initial condition \eqref{ms2.7}. By the delta product rule,
$$ \widehat{U}^\Delta(t)=\widehat{Y}^{\sigma}(t)\mathscr{M}^\Delta(t) + \widehat{Y}^\Delta(t)\mathscr{M}(t). $$
Multiply by $J$, use \eqref{ms2.3}, and assume $\widehat{U}$ satisfies \eqref{ms5.3} to obtain
$$ J\widehat{Y}^{\sigma}(t)\mathscr{M}^\Delta(t) = \widetilde{H}^*(t)A(t)f(t). $$
Then by \eqref{ms2.5} and Lemma \ref{mslem2.4} we have 
\begin{eqnarray*}
 \mathscr{M}^\Delta(t) &=& -(\widehat{Y}^{\sigma})^{-1}(t)J\widetilde{H}^*(t)A(t)f(t) \\
 &=& -J\widehat{Z}^{\sigma*}(t)\widetilde{H}^*(t)A(t)f(t)  \\
 &=& -JZ^*(t)A(t)f(t) \\
 &=& \begin{pmatrix} 0_n & I_n \\ 0_n & M(\lambda) \end{pmatrix}Z^*(t)A(t)f(t) - \begin{pmatrix} 0_n & 0_n \\ I_n & M(\lambda) \end{pmatrix}Z^*(t)A(t)f(t).
\end{eqnarray*}
Consequently we have, up to addition of a constant matrix,
\begin{eqnarray*}
 \mathscr{M}(t) &=& \int_{t_0}^{t} \left(\begin{smallmatrix} 0_n & I_n \\ 0_n & M(\lambda) \end{smallmatrix}\right)Z^*(s)A(s)f(s)\Delta s +  \int_{t}^{\infty} \left(\begin{smallmatrix} 0_n & 0_n \\ I_n & M(\lambda) \end{smallmatrix}\right)Z^*(s)A(s)f(s)\Delta s \\
 &=& \int_{t_0}^{t} \left(\begin{smallmatrix} I_n \\ M(\lambda) \end{smallmatrix}\right)\chi^*(s)A(s)f(s)\Delta s + \int_{t}^{\infty} \left(\begin{smallmatrix} 0_n \\ \eta^*(s)+M(\lambda)\chi^*(s) \end{smallmatrix}\right)A(s)f(s)\Delta s,
\end{eqnarray*}
since $Z=(\eta|\chi)$, so that
\begin{eqnarray}
 \widehat{U}(t) &=& \widehat{Y}(t)\mathscr{M}(t) = \left(\widehat{\theta}(t)|\widehat{\phi}(t)\right)\mathscr{M}(t) \nonumber \\
 &=& \widehat{\psi}(t)\int_{t_0}^{t}\chi^*(s)A(s)f(s)\Delta s + \widehat{\phi}(t)\int_{t}^{\infty} \zeta^*(s)A(s)f(s)\Delta s, \label{Uhatsol}
\end{eqnarray}
where $\psi$ and $\zeta$ are as defined in Theorem \ref{thm4.1}. It follows from \eqref{ms2.1} and \eqref{ms5.2} that
\begin{eqnarray}
 U(t) &=& H(t)\widehat{U}(t) + N(t)A(t)f(t) \nonumber \\
 &=& \psi(t)\int_{t_0}^{t}\chi^*(s)A(s)f(s)\Delta s + \phi(t)\int_{t}^{\infty} \zeta^*(s)A(s)f(s)\Delta s + N(t)A(t)f(t) \nonumber \\
 &=& (R_{\lambda}f)(t). \label{UdefRlam}
\end{eqnarray}
We show that $U$ solves \eqref{ms5.1}. Using \eqref{Uhatsol}, we have (supressing the $t$)
\begin{eqnarray*}
 J\widehat{U}^\Delta &=& J\left[\widehat{\psi}^\sigma \chi^*-\widehat{\phi}^\sigma\zeta^*\right]Af + J\left[\widehat{\psi}^\Delta\int_{t_0}^t\chi^*(s)A(s)f(s)\Delta s +\widehat{\phi}^\Delta\int_{t}^{\infty} \zeta^*(s)A(s)f(s)\Delta s \right] \\
 &=& J\left[\widehat{\theta}^\sigma \chi^*-\widehat{\phi}^\sigma\eta^*\right]Af + (\lambda A+B)(U-NAf) \\
 &=& -J\widehat{Y}^\sigma JZ^*Af + (\lambda A+B)(U-NAf) \\
 &=& (\lambda A+B)U + Af,
\end{eqnarray*}
using \eqref{ms2.5}, Lemma \ref{mslem2.4} and the fact that $(\lambda A+B)N+I_{2n}=\widetilde{H}^*$ by \eqref{Ndef}.

Continuing from \eqref{ms2.7}, \eqref{Uhatsol} and \eqref{UdefRlam}, we have
$$ \widehat{(R_{\lambda}f)}(t_0)=\left(\widehat{R}_{\lambda}f\right)(t_0) = \widehat{\phi}(t_0,\lambda)\int_{t_0}^{\infty} \zeta^*(s)A(s)f(s)\Delta s = \begin{pmatrix} - \displaystyle\int_{t_0}^{\infty} \zeta^*(s)A(s)f(s)\Delta s \\ 0_n \end{pmatrix}, $$
so that we have $(R_\lambda f)_2^\rho(t_0)=0_n$. The identity 
$$ \widehat{\chi}^*(t_0)J\left(\widehat{R}_{\lambda}f\right)(t_0) = \widehat{\chi}^*(t_0)J\widehat{\phi}(t_0)\int_{t_0}^{\infty}\zeta^*(t)A(t)f(t) \Delta t = 0 $$
follows immediately from the initial values of $\chi$ and $\phi$ in \eqref{ms2.7}. Note that the integral is convergent as each column of $\zeta$ is $A-$square integrable and $f\in L_A^2(t_0,\infty)_\T$. In addition, since $\widehat{\psi}(t)=\widehat{Y}(t)\left(\begin{smallmatrix}I_n \\ M \end{smallmatrix}\right)$ and $\widehat{\zeta}(t)=\widehat{Z}(t)\left(\begin{smallmatrix}I_n \\ M^* \end{smallmatrix}\right)$ we obtain from Lemma \ref{mslem2.4} the equalities
$$ \widehat{\zeta}^*(t,\lambda)J\widehat{\psi}(t,\lambda) = (I_n|M)\widehat{Z}^*(t,\lambda)J\widehat{Y}(t,\lambda)\left(\begin{smallmatrix}I_n \\ M \end{smallmatrix}\right) = 0_n $$
and
\begin{equation}\label{ms5.6}
 \widehat{\zeta}^*(t,\lambda)J\widehat{\phi}(t,\lambda) = (I_n|M)\widehat{Z}^*(t,\lambda)J\widehat{Y}(t,\lambda)\left(\begin{smallmatrix}0_n \\ I_n \end{smallmatrix}\right) = -I_n.
\end{equation}
It thus follows that
$$ \lim_{t\rightarrow\infty} \widehat{\zeta}^*(t,\lambda)J\left(\widehat{R}_{\lambda}f\right)(t) = -\lim_{t\rightarrow\infty} \int_{t}^{\infty} \zeta^*(s)A(s)f(s)\Delta s = 0, $$
completing the proof.
\end{proof}


\begin{remark}
Note that 
$$ \widehat{(R_{\lambda}f)}(t)=\left(\widehat{R}_{\lambda}f\right)(t)= H^{-1}(t)\left[(R_{\lambda}f)-NAf\right](t) $$ 
for $t\in[t_0,\infty)_\T$ by \eqref{ms5.2}. For the adjoint system \eqref{ms5.4}, we may show in a similar manner that
$$ J\left(\widetilde{R}_{\lambda}f\right)^\Delta(t) = \left(\overline{\lambda}A(t)+B^*(t)\right)\left(\widetilde{R}_{\lambda}f\right)(t) + A(t)f(t), $$
with boundary conditions $\widehat{\phi}^*(t_0)J\left(\widehat{\widetilde{R}}_{\lambda}f\right)(t_0)=0$ and $\lim_{t\rightarrow\infty}\widehat{\psi}^*(t,\lambda)J\left(\widehat{\widetilde{R}}_{\lambda}f\right)(t)=0$. 
\end{remark}


\begin{theorem}
Let $\lambda\in\Lambda(\lambda_0,\mathscr{U}_{2n})\cap\widetilde{\Lambda}(\lambda_0,\mathscr{U}^{-1}_{2n})$ and $f\in L_A^2(t_0,\infty)_\T$. If we set $\Phi=R_{\lambda}f$ and $\widetilde{A}(t) = \mathscr{U}_{2n}A(t)\mathscr{U}^*_{2n}$, then
$$ \|\Phi\|^2_{W(\lambda_0)} + (\delta-\varepsilon) \|\Phi\|^2_{\widetilde{A}} \;\le\; \frac{1}{4\varepsilon}\|f\|_A^2 $$
for any $0<\varepsilon<\delta$, with $\delta=\delta(\lambda)$ as in \eqref{ms3.3}, and
$$ \|\Phi\|^2_{\widetilde{A}}\;\le\;\frac{1}{\delta}\|f\|_{A}. $$
In particular, since $\widetilde{A}(t)\asymp A(t)$ for $t\in[t_0,\infty)_\T$, $R_{\lambda}$ is bounded on $L_A^2(t_0,\infty)_\T$.
\end{theorem}

\begin{proof}
The proof is similar to that given in the continuous case \cite[Theorem 5.1]{b2} and the discrete case \cite[Theorem 5.3]{ms}, and thus is omitted.
\end{proof}


\begin{lemma}
If $\lambda\in\Lambda(\lambda_0,\mathscr{U}_{2n})$, then $R_{\lambda}$ is an injective linear operator defined on $L_A^2(t_0,\infty)_\T$. In addition, for any $\xi\in\Lambda(\lambda_0,\mathscr{U}_{2n})$ and $L_\xi$ as in \eqref{Lxi}, we have $\range R_{\xi}\subset \mathscr{D}(L_\xi)$ and $(L_{\xi}-\xi)R_{\xi} f=f$ for all $f\in L_A^2(t_0,\infty)_\T$. A similar statement holds for $\widetilde{R}_{\xi}$.
\end{lemma}

\begin{proof}
See the proof in \cite[Lemma 5.4]{ms}.
\end{proof}


\begin{remark}
In the following lemma, it is traditional to use the Greek letter $\rho$ to represent a resolvent set, but since we already employ $\rho$ as a backward jump operator on time scales, we will use $r$ instead.
\end{remark}


\begin{lemma}
Denoting the resolvent sets of $L_\xi$ and $\widetilde{L}_\xi$ by $r(L_\xi)$ and $r\left(\widetilde{L}_\xi\right)$, respectively, we have $\xi\in r(L_\xi)$, $\overline{\xi}\in r\left(\widetilde{L}_\xi\right)$, $\range R_{\xi}=\mathscr{D}(L_\xi)$, $\range\widetilde{R}_{\xi}=\mathscr{D}\left(\widetilde{L}_\xi\right)$, $(L_{\xi}-\xi)^{-1}=R_{\xi}$ and $\left(\widetilde{L}_{\xi}-\overline{\xi}\right)^{-1}=\widetilde{R}_{\xi}$. In addition, $L_\xi$ and $\widetilde{L}_\xi$ are closed.
\end{lemma}

\begin{proof}
See \cite[Lemma 5.7]{b2} and \cite[Lemma 5.5]{ms}.
\end{proof}


\begin{lemma}
The space $\mathscr{D}(L_\xi)$ is dense in $L_A^2(t_0,\infty)_\T$. Also, $\widetilde{L}_{\xi}=L_{\xi}^*$, the adjoint of $L_{\xi}$.
\end{lemma}

\begin{proof}
See \cite[Lemma 5.8]{b2} and \cite[Lemma 5.6]{ms}.
\end{proof}


\begin{theorem}
If $\xi\in\Lambda(\lambda_0,\mathscr{U}_{2n})$, then $\Lambda(\lambda_0,\mathscr{U}_{2n})\subset r(L_{\xi})$, and
$$ (L_{\xi}-\lambda)^{-1}=R_{\lambda}, \quad \lambda\in\Lambda(\lambda_0,\mathscr{U}_{2n}). $$
A corresponding statement holds for $\widetilde{R}_{\lambda}$ and $\widetilde{L}_{\xi}$. 
\end{theorem}

\begin{proof}
See \cite[Theorem 5.9]{b2} and \cite[Theorem 5.7]{ms}.
\end{proof}


\begin{theorem}
If all solutions of \eqref{maineq} and \eqref{ms2.4} are $A-$square integrable for some $\lambda'\in\C$, and if
$$ A^{1/2}(t)N^*(t)A(t)N(t)A^{1/2}(t)\rightarrow 0 \quad\text{as}\quad t\rightarrow\infty $$
for $A$ in \eqref{ABdef} and $N$ in \eqref{Ndef}, then all solutions of \eqref{maineq} are $A-$square integrable for all $\lambda\in\C$.
\end{theorem}

\begin{proof}
Note that by a modified variation of constants approach (see \cite[Lemma 6.5]{b2}, \cite[Chapter 9 Theorem 2.1]{cod}, \cite[Theorem A.1]{ms}), any solution of
$$ J\widehat{y}^\Delta(t) = \Big(\lambda A(t)+B(t)\Big) y(t) = \Big(\lambda' A(t)+B(t)\Big) y(t) + (\lambda-\lambda')A(t)y(t), \quad t\in[t_0,\infty)_\T, $$
can be written as 
\begin{eqnarray*} 
 y(t,\lambda) &=& 
   \theta(t,\lambda')\gamma+\phi(t,\lambda')\widetilde{\gamma} \\ 
   & & +(\lambda-\lambda')\left(\int_{c}^{t}\left(\psi(t,\lambda')\chi^*(s,\lambda') 
       -\phi(t,\lambda')\zeta^*(s,\lambda')\right)A(s)y(s,\lambda) \Delta s+N(t)A(t)y(t,\lambda)\right)
\end{eqnarray*}
for some $\gamma,\widetilde{\gamma}\in\C^n$, where $N$ is given in \eqref{Ndef}. The remainder of the proof is similar to that given in \cite[Theorem A.1]{ms} and is omitted.
\end{proof}


\end{document}